\documentclass[reqno]{amsart}

\usepackage{amsmath}
\usepackage{amssymb}
\usepackage{amsthm} 
\usepackage{graphicx}
\usepackage{caption}
\usepackage{cite}
\usepackage{enumerate}
\usepackage{mathtools}

\newtheorem{Thm}{Theorem}

\newtheorem{Lemma}[Thm]{Lemma}
\newtheorem{Cor}[Thm]{Corollary}

\theoremstyle{definition}

\def\bkappa{\boldsymbol\kappa}
\def\btau{\boldsymbol{\tau}}

\def\bw{\mathbf{w}}

\def\intst{\int_{\St}}

\def\nnu{\boldsymbol{\nu}}
\def\pa{\partial}
\def\R{\mathbb{R}}
\def\Si{\Sigma}
\def\Sp{\Sph^1}
\def\Sph{\mathbb{S}}
\def\St{\Si_t}

\def\Z{\mathbb{Z}}

\newcommand{\fracd}[2]{{\frac{d #1}{d #2}}}
\newcommand{\fracp}[2]{{\frac{\pa #1}{\pa #2}}}

\DeclareRobustCommand{\SkipTocEntry}[4]{}

\begin{document}

\title{A distance comparison principle for curve flows with a global forcing term}
\author{Friederike Dittberner}
\email{dittberner@math.fu-berlin.de}

\begin{abstract}
We consider closed, embedded, smooth curves in the plane and study their behaviour under curve flows with a global forcing term.
We prove an analogue to Huisken's distance comparison principle for curve shortening flow for initial curves whose local total curvature does not lie below $-\pi$ and arbitrary global forcing terms.
\end{abstract}

\maketitle
\tableofcontents

\section{Introduction}

Let $\Si_0\subset\R^2$ be a closed, embedded, smooth and positively oriented curve, parametrised by the embedding $X_0:\Sp\to\R^2$.
Let $X:\Sp\times[0,T)\to\R^2$ be a one-parameter family of maps with $X(\,\cdot\,,0)=X_0$ satisfying the evolution equation
\begin{align}\label{eq:ccf}
\fracp{X}{t}(p,t)=\big(h(t)-\kappa(p,t)\big)\nnu(p,t)
\end{align}
for $(p,t)\in\Sp\times(0,T)$, where the vector $\nnu$ is the outward pointing unit normal to the curve $\St:=X(\Sp,t)$, $\kappa$ is the curvature function and $T$ is the maximal time of existence.
The global term 
\begin{align}\label{eq:h_1}
h(t)\in[0,\infty)
\end{align}
is smooth and smoothly bounded whenever the curvature is bounded.\\

The curve shortening flow with $h\equiv0$ was first studied by Gage--Hamilton and Grayson~\cite{GageHamilton86,Grayson87}.
The enclosed area preserving curve shortening flow with $h(t)=2\pi/L_t$, where $L_t=L(\St)$ is the length of the curve, was introduced by Gage~\cite{Gage86} and the length preserving curve flow with $h(t)=\intst\kappa^2\,ds_t/2\pi$ by Pihan~\cite{Pihan98}.
In~\cite{JianPan08}, Jian and Pan consider $h(t)=L_t/2A_t$, where $A_t=A(\St)$ is the enclosed area of the curve.\\

In this paper, we prove a distance comparison principle for closed, embedded, smooth initial curves with
\begin{align*}
\int_p^q\kappa\,ds\geq-\pi
\end{align*}
for all $p,q\in\Sp$, evolving under~\eqref{eq:ccf} with global term~\eqref{eq:h_1}.
With this, we extend our result from~\cite{Dittberner21}[Thm.~5.3 and Cor.~5.4]. \\

\textbf{Acknowledgements.}
We want to thank Laiyuan Gao from Jiangsu Normal University for pointing out an error in the proof of~\cite{Dittberner21}[Thm.~5.3].
The error lies in part (i) of the proof, specifically in the equation above equation~(5.8). 
Affected are equations~(5.8) and~(5.16).
In this paper we will correct the error and improve the result of the theorem.

\section{Preliminaries}

Let $X:\Sp\to\R^2$ be a smooth, embedded curve with length element $v(p):=\big\Vert\fracd{}{p}X(p)\big\Vert$.
For a fixed point $p_0\in\Sp$, the arc length parameter is given by $s(p):=\int_{p_0}^pv(r)\,dr$, so that $ds=vdp$ and $\fracd{}{s}=\frac1{v}\fracd{}{p}$.
The map $\tilde X:=X\circ s^{-1}:\Sp_{L/2\pi}\to\R^2$ parametrises $\Si$ by arc length.
The unit tangent vector field $\btau$ to $\Si$ in direction of the arc length parametrisation is given by $\btau:=\fracd{}{s}\tilde X$ and the outward unit normal by $\nnu:=(\btau_2,-\btau_1)$.
We define the curvature by $\kappa:=-\langle\fracd{}{s}\btau,\nnu\rangle=\langle\btau,\fracd{}{s}\nnu\rangle$ and the curvature vector by $\bkappa:=-\kappa\nnu$.\\

We define the extrinsic and intrinsic distance $d,l:\Sp\times\Sp\times[0,T)\to\R$ by
$$d(p,q,t):=\Vert X(q,t)-X(p,t)\Vert_{\R^2}\quad\text{ and }\quad l(p,q,t):=\int_p^qds_t$$
and the vector $\bw:\big(\Sp\times\Sp\times[0,T)\big)\setminus\{d=0\}\to\R^2$ by
$$\bw(p,q,t):=\frac{X(q,t)-X(p,t)}{d(p,q,t)}\,.$$
Like in~\cite{Huisken95}, we define $\psi:\Sp\!\times\Sp\!\times[0,T)\to\R$ by
\begin{align}\label{eq:defpsi}
\psi(p,q,t):=\frac{L_t}\pi\sin\!\left(\frac{\pi l(p,q,t)}{L_t}\right)\,,
\end{align}
where $L_t:=L(\St)<\infty$ is the length of the curve.
Lastly, we define the total local curvature $\theta:\Sp\times\Sp\times[0,T)\to\R$ by
\begin{align}\label{eq:deftheta}
\theta(p,q.t):=\int_p^q\kappa(r,t)\,ds_t\,,
\end{align}
where we integrate in direction of the parametrisation.

\begin{Lemma}[Dittberner~{\cite{Dittberner21}[Lem.~3.2]}]\label{lem:minmaxtheta}
Let $\Si=X(\Sp)$ be an embedded, closed curve. 
Then $\sup_{\Sp\!\times\Sp}\theta=2\pi-\min_{\Sp\!\times\Sp}\theta$.
\end{Lemma}

\begin{Thm}[Dittberner~{\cite{Dittberner21}[Thm.~3.4]}]\label{thm:dtdstheta}
Let $X:\Sp\times(0,T)\to\R^2$ be a solution of~\eqref{eq:ccf}.
Then $(\fracp{}{t}-\Delta_{\St})\theta(p,q,t)=0$ for all $p,q\in\Sp$, $p\ne q$, and $t\in(0,T)$.
Moreover, let $t_0\in(0,T)$ and suppose $\theta_{\min}(t_0)<0$. 
Then $\theta_{\min}(t_0)<\theta_{\min}(t)$ for all $t\in(t_0,T)$.
\end{Thm}

\begin{Lemma}[Dittberner~{\cite{Dittberner21}[Lem.~3.5]}]\label{lem:thetabeta}
Let $\Si=X(\Sp)$ be an embedded curve and $p,q\in\Sp$ with $d(p,q)\ne0$.
Let $\langle\bw,\btau_p\rangle=\langle\bw,\btau_q\rangle=\cos(\beta/2)$ for $\beta\in[0,\pi]$.
Then either
\begin{enumerate}[(i)]
\item $\langle\bw,\nnu_p\rangle=-\langle\bw,\nnu_q\rangle=-\sin(\beta/2)$ and $\theta(p,q)=2\pi k+\beta$,
\item $\langle\bw,\nnu_p\rangle=-\langle\bw,\nnu_q\rangle=\sin(\beta/2)$ and $\theta(p,q)=2\pi k-\beta$, or
\item $\langle\bw,\nnu_p\rangle=\langle\bw,\nnu_q\rangle=\pm\sin(\beta/2)$ and $\theta(p,q)=2\pi k$
\end{enumerate}
for $k\in\Z$.
\end{Lemma}

\section{Distance comparison principle}

In the following, we study the flow~\eqref{eq:ccf} for embedded, positively oriented, smooth initial curves $\Si_0=X_0(\Sp)$ with
\begin{align}\label{eq:intkappageqminuspi}
\theta_0(p,q)=\int_p^q\kappa\,ds\geq-\pi
\end{align}
for all $p,q\in\Sp$. 
We adapt the methods from Huisken~\cite{Huisken95} and extend the results of~\cite{Dittberner21}.

\begin{Thm}
\label{thm:minimumdpsi}
Let $\Si_0=X_0(\Sp)$ be a smooth, embedded curve satisfying~\eqref{eq:intkappageqminuspi}. 
Let $X:\Sp\!\times[0,T)\to\R^2$ be a solution of~\eqref{eq:ccf} satisfying~\eqref{eq:h_1} and with initial curve $\Si_0$.
Then there exists a constant $c(\Si_0)>0$ such that
$$\inf_{(p,q,t)\in\Sp\!\times\Sp\!\times[0,T)}\frac d\psi(p,q,t)\geq c\,.$$
\end{Thm}

\begin{proof}
Except for section~(i), the proof is the same as the one for~\cite{Dittberner21}[Thm~5.3].
For the convenience of the reader, we sketch the whole proof here.\\

Let $\Si_0=X_0(\Sp)$ be an embedded closed curve satisfying~\eqref{eq:intkappageqminuspi}. 
Lemma~\ref{lem:minmaxtheta} and Theorem~\ref{thm:dtdstheta} imply that $\theta_0\in[-\pi,3\pi]$ and 
\begin{align}\label{eq:thetainteral2}
\theta(p,q,t)\in(-\pi,3\pi)
\end{align}
for all $p,q\in\Sp$ and $t\in(0,T)$. 
Since $d/\psi$ is continuous and initially positive, there exists a time $T'\in(0,T]$ so that $d/\psi>0$ on $[0,T')$.
Fix $t_0\in(0,T')$.
Assume that $\Si_{t_0}$ is not a circle so that $\min_{\Sp\!\times\Sp}(d/\psi)<1$ at $t_0$.
Let $p,q\in\Sp$, $p\ne q$, be points where a local spatial minimum of $d/\psi$ at $t_0$ is attained and assume w.l.o.g.\ that $l(p,q,t_0)\leq L/2$, where $L:=L_{t_0}$. 
Like in~\cite{Huisken95}, we get from the first spatial derivatives of $d/\psi$ that there exists $\beta\in(0,\pi]$ with
\begin{align}\label{eq:wtaucostheta}
\langle\bw,\btau_p\rangle=\langle\bw,\btau_q\rangle
=\frac d\psi\cos\!\left(\frac{\pi l}{L}\right)
=\cos\!\left(\frac\beta2\right)
\in[0,1)\,.
\end{align}
By Lemma~\ref{lem:thetabeta} and~\eqref{eq:thetainteral2} it is sufficient to consider the following three different cases.\\

(i) Assume that $\langle\bw,\nnu_p\rangle=-\langle\bw,\nnu_q\rangle=-\sin(\beta/2)$ and $\beta=\theta\in(0,\pi]$.
Then\begin{align}\label{eq:sinthetasinpilL}
\langle\bw,\nnu_q-\nnu_p\rangle
=2\sin\!\left(\frac\theta2\right)\,.
\end{align}
By~\eqref{eq:wtaucostheta},
\begin{align}\label{eq:costhetacosdpsi}
\cos\!\left(\frac\theta2\right)
=\frac d\psi\cos\!\left(\frac{\pi l}{L}\right)
<\cos\!\left(\frac{\pi l}{L}\right)\,.
\end{align}
As $\pi l/\!L\in(0,\pi/2]$ and the cosine function is axially symmetric and strictly decreasing on $(0,\pi/2]$,~\eqref{eq:costhetacosdpsi} implies $\theta/2>\pi l/L$. 
By the definition~\eqref{eq:deftheta} of $\theta$, then
\begin{align}\label{eq:thetah}
-h\int_p^q\kappa\,ds_t+\int_p^q\kappa^2\,ds_t
\geq-h\theta+\frac{\theta^2}l
>-h\theta+\frac{4\pi^2l}{L^2}.
\end{align}
Like in~\cite{Huisken95}, we get from the second spatial derivatives of $d/\psi$ and~\eqref{eq:defpsi} that
\begin{align}\label{eq:wkappasintheta}
\langle\bw,\bkappa_q-\bkappa_p\rangle
\geq-\frac{4\pi^2d}{L^2}
=-\frac{4\pi d}{\psi L}\sin\left(\frac{\pi l}{L}\right)\,.
\end{align} 
We use the evolution equation~\eqref{eq:ccf} to differentiate the ratio in time (for details see also~\cite{Dittberner21}[Lem.~2.1]) and obtain by~\eqref{eq:sinthetasinpilL},~\eqref{eq:costhetacosdpsi},~\eqref{eq:thetah} and~\eqref{eq:wkappasintheta},
\begin{align*}
&\fracp{}{t}\!\left(\frac d\psi\right) \notag\\
&\;=\frac1\psi\big(h\langle\bw,\nnu_q-\nnu_p\rangle+\langle\bw,\bkappa_q-\bkappa_p\rangle\big) 
		-\frac d{\psi^2}\cos\!\left(\frac{\pi l}{L}\right)
		\left(h\int_p^q\kappa\,ds_t-\int_p^q\kappa^2\,ds_t\right) \notag\\
	&\;\quad+\frac d{\pi\psi^2}\left(\intst\kappa^2\,ds_t-2\pi h\right)
			\left(\sin\!\left(\frac{\pi l}{L}\right)-\frac{\pi l}{L}\cos\!\left(\frac{\pi l}{L}\right)\right)\notag\\
&\;\geq\frac{2h}\psi\sin\!\left(\frac\theta2\right) 
		-\frac{4\pi d}{\psi^2 L}\sin\left(\frac{\pi l}{L}\right)
		-\frac d{\psi^2}\cos\!\left(\frac{\pi l}{L}\right)
		\left(h\theta-\frac{4\pi^2l}{L^2}\right)\notag\\
	&\;\quad+\frac{d}{\psi^2}\left(\frac1{\pi}\intst\kappa^2\,ds_t-2h\right)
			\left(\sin\!\left(\frac{\pi l}{L}\right)-\frac{\pi l}{L}\cos\!\left(\frac{\pi l}{L}\right)\right)\notag\\
&\;\geq\frac{2h}\psi\sin\!\left(\frac\theta2\right) 
		-\frac{h\theta}\psi\cos\!\left(\frac\theta2\right)\notag\\
	&\;\quad+\frac{d}{\psi^2}\left(\frac1{\pi}\intst\kappa^2\,ds_t-2h-\frac{4\pi}{L}\right)
			\left(\sin\!\left(\frac{\pi l}{L}\right)-\frac{\pi l}{L}\cos\!\left(\frac{\pi l}{L}\right)\right)\notag\\
&\;\geq\frac{2h}\psi\left(\sin\!\left(\frac\theta2\right)-\frac\theta2\cos\!\left(\frac\theta2\right)
		-\frac d\psi\left(\sin\!\left(\frac{\pi l}{L}\right)-\frac{\pi l}{L}\cos\!\left(\frac{\pi l}{L}\right)\right)\right)\notag\\
	&\;\quad+\frac{d}{\psi^2}\left(\frac1{\pi}\intst\kappa^2\,ds_t-\frac{4\pi}{L}\right)
			\left(\sin\!\left(\frac{\pi l}{L}\right)-\frac{\pi l}{L}\cos\!\left(\frac{\pi l}{L}\right)\right)
			>0\,,
\end{align*}
at $(p,q,t_0)$. 
For the last inequality, we used that $d/\psi\in(0,1)$, $\intst\kappa^2\,ds_t>4\pi^2/L$, $\theta/2>\pi l/L$ and that $\sin(x)-x\cos(x)$ is positive and strictly increasing on $(0,\pi]$.\\

(ii) Assume that $\langle\bw,\nnu_p\rangle=-\langle\bw,\nnu_q\rangle=\sin(\beta/2)$ and $-\beta=\theta\in(-\pi,0)$.
Like in~\cite{Dittberner21}[Thm~5.3], we deduce that $d/\psi\geq c(\Si_0)>0$ at $(p,q,t_0)$. \\

(iii) Assume that $\theta\in\{0\}\cup(\pi,3\pi)$ or $\theta=\pi$ and $\langle\bw,\nnu_p\rangle=-\langle\bw,\nnu_q\rangle=1$. 
Again, like in~\cite{Dittberner21}[Thm~5.3], we arrive at $d/\psi>\min_{\Sp\!\times\Sp}d/\psi(\,\cdot\,,\,\cdot\,,t_0)$,
where $\min_{\Sp\!\times\Sp}(d/\psi)(\,\cdot\,,\,\cdot\,,t_0)$ is attained at a point which was treated in cases~(i) and~(ii). \\

Assume that $d/\psi$ falls below $c$ and attains $\Lambda\in(0,c)$ for the first time at time $t_2\in(0,T)$ and points $p,q\in\Sp$, $p\ne q$, so that
\begin{align}\label{eq:c1mindpsi}
c>\Lambda=\frac d\psi(p,q,t_2)=\min_{\Sp\!\times\Sp}\frac d\psi(\,\cdot\,,\,\cdot\,,t_2)
\end{align}
is a global minimum and
\begin{align}\label{eq:dtdpsileq0}
\fracp{}{t}_{|_{t=t_2}}\!\left(\frac d\psi\right)(p,q,t)\leq0\,.
\end{align}
Case~(i) contradicts~\eqref{eq:dtdpsileq0}, and cases~(ii) and~(iii) contradict~\eqref{eq:c1mindpsi}. 

\end{proof}

\begin{Cor}
\label{cor:embeddedness}
Let $\Si_0=X_0(\Sp)$ be a smooth, embedded curve satisfying~\eqref{eq:intkappageqminuspi}.
Let $X:\Sp\!\times[0,T)\to\R^2$ be a solution of~\eqref{eq:ccf} satisfying~\eqref{eq:h_1} and with initial curve $\Si_0$.
Then $\St=X(\Sp\!,t)$ is embedded for all $t\in(0,T)$.
\end{Cor}

\providecommand{\bysame}{\leavevmode\hbox to3em{\hrulefill}\thinspace}
\providecommand{\MR}{\relax\ifhmode\unskip\space\fi MR }
\providecommand{\MRhref}[2]{%
  \href{http://www.ams.org/mathscinet-getitem?mr=#1}{#2}
}
\providecommand{\href}[2]{#2}

\bibliographystyle{amsplain} 
\end{document}